\documentclass{amsart}
\newtheorem{thm}{Theorem}[section]
\newtheorem{cor}[thm]{Corollary}
\newtheorem{lem}[thm]{Lemma}

\newtheorem{exa}[thm]{Example}

\theoremstyle{definition}
\newtheorem{defn}[thm]{Definition}
\theoremstyle{remark}
\newtheorem{rem}[thm]{Remark}

\numberwithin{equation}{section}

\begin{document}
\vbox{\vskip 3.5cm}
\begin{center}
{\bf {\large Weakly singular integral inequalities and global solutions for fractional differential equations of  Riemann-Liouville type}}
\end{center}
\title{}
\maketitle

\bigskip

\begin{center}
Tao Zhu\\
\bigskip
Department of Mathematics and Physics, Nanjing
Institute of Technology,\\
Nanjing, 211100, P. R. China \\
\end{center}

\baselineskip=20pt

\bigskip

\bigskip
\begin{abstract}
In this paper, we obtain some new results about weakly singular integral inequalities. These inequalities are used
to discuss the global existence and uniqueness results for fractional differential equations of  Riemann-Liouville type. Some examples are provided to illustrate the applicability of our main results.
\end{abstract}
\bigskip

\section{Introduction}
\bigskip
In 1981, Henry [6] studied the following linear integral
inequality
\begin{equation}
u(t)\leq a(t)+b\int_{0}^{t}(t-s)^{\beta-1}u(s)ds.
\end{equation}
Another version of a weakly singular result of Henry [6] is the following
\begin{equation}
u(t)\leq
at^{\alpha-1}+b\int_{0}^{t}(t-s)^{\beta-1}s^{\gamma-1}u(s)ds.
\end{equation}
Since weakly singular integral inequalities are a well-known tool in proving the existence, uniqueness and stability of integral equations, evolution equations and fractional differential equations, many scholars have began to study weakly singular integral inequalities and obtained a number of versions of weakly singular integral inequalities (for example, see [2,4,5,6,8,9,11,15,17,18,19,20]). Recently, Medve\v{d} [8,9] studied the following nonlinear integral inequality
\begin{equation}
u(t)\leq
a(t)+\int_{0}^{t}(t-s)^{\beta-1}l(s)\omega(u(s))ds,
\qquad t\in[0,T) (0<T\leq+\infty).
\end{equation}
Zhu [19] used the same method as in [8,9] to improve the proof of the inequality (1.3).
By a new method, Zhu [20] investigated the following integral inequality
\begin{equation}
u(t)\leq a(t)+b(t)(\int_{0}^{t}l^{p}(s)u^{p}(s)ds)^{\frac{1}{p}},
\qquad t\in[0,\infty).
\end{equation}
The inequality (1.4) was first studied by Willett in [16] by using the Minkowski inequality.
The aim of this paper is to continue the research in this field.
We first deal with the following nonlinear integral
inequality
\begin{equation}
u(t)\leq a(t)+b(t)(\int_{0}^{t}l(s)\omega(u(s))ds)^{\frac{1}{p}},
\qquad t\in[0,T).
\end{equation}
As far as I know, there have been few papers to study the integral inequality (1.5). Then using the integral inequality (1.5), we study the following weakly singular integral inequality
\begin{equation}
u(t)\leq
a(t)+b(t)\int_{0}^{t}(t-s)^{\beta-1}l(s)\omega(u(s))ds,
\qquad t \in[0,T),
\end{equation}
and the weakly singular integral inequality
\begin{equation}
u(t)\leq
at^{-\alpha}+bt^{-\delta}\int_{0}^{t}(t-s)^{\beta-1}f(s,u(s))ds,
\qquad t\in(0,T),
\end{equation}
where $a, b\geq 0$, $\alpha>\delta\geq 0$ and $0<\beta<1$.

As applications of our results, we present the existence and uniqueness results
of the following fractional differential equation
\begin{equation}
\begin{cases}
D_{r}^{\beta}x(t)=f(t,x(t))  \qquad \beta \in(0,1),\\
\lim_{t\rightarrow 0^{+}}t^{1-\beta}x(t)=x_{0},
\end{cases}
\end{equation}
where $D_{r}^{\beta}$ is the Riemann-Liouville fractional
derivative. In the recent years, many scholars have devoted to the study of the fractional differential
equation (1.8). For example, Trif [14] studied the global existence of solutions for the fractional differential
equations (1.8)  when
$f(t,x)\leq p(t)|x|+q(t)$, where $p\in C_{\alpha}(0,+\infty)$ and
$q\in C_{1-\beta}(0,+\infty)$ with $0\leq\alpha<\beta$ and
$2\beta-\alpha>1$. Webb [15] obtained the existence results of the fractional differential equation (1.8) under the assumption of nonnegative function $f(t,x)=t^{-\gamma}g(t,x)$, where $g(t,x)\leq M(1+x)$, $M>0$ and $0\leq\gamma<\beta$. Zhu [20] presented the existence results
of the fractional differential equation (1.8) when $|f(t, x)|\leq l(t)|x|+k(t)$ with $t^{\beta-1}l(t)\in C(0,+\infty)\bigcap L_{Loc}^{p}[0,+\infty)$ and
$k(t)\in C(0,+\infty)\bigcap L_{Loc}^{p}[0,+\infty)$ ($p>\frac{1}{\beta}$).

In this paper, by fixed point theorem and weakly singular integral inequality, we prove the existence of global solutions of the fractional differential
equation (1.8) when $|f(t, x)|\leq l(t)\omega(t^{1-\beta}|x|)$, where $l(t)\in C(0,+\infty)$ with $t^{1-\beta}l(t)\in L^{p}_{Loc}[0,+\infty)$ ($p>\frac{1}{\beta}$) and nonnegative
nondecreasing function $\omega\in C[0,+\infty)$ with $\lim_{t\rightarrow+\infty}\frac{t}{\omega(t)}=K (0<K\leq+\infty)$. Applying the above conclusion, we can easily obtain the following results. If $|f(t,x)|\leq l(t)|x|^{\gamma}+k(t)$, where $t^{(1-\gamma)(1-\beta)}l(t)\in
C(0,+\infty)\bigcap L^{p}_{Loc}[0,+\infty)$ and $t^{1-\beta}k(t)\in
C(0,+\infty)\bigcap L^{p}_{Loc}[0,+\infty)$$(0<\gamma\leq1, p>\frac{1}{\beta})$, then the fractional differential equation (1.8) has at least one
global solution in $C_{1-\beta}(0,+\infty)$. We also prove a unique solution of the fractional differential
equation (1.8) when $|f(t,x)-f(t,y)|\leq l(t)|x-y|$, where $l(t)\in
C(0,+\infty)\bigcap L^{p}_{Loc}[0,+\infty)$ and $t^{1-\beta}|f(t,0)|\in
C(0,+\infty)\bigcap L^{p}_{Loc}[0,+\infty)$$(p>\frac{1}{\beta})$. Our results improve and generalize some of the results in [14,15,20]. Finally, we will give some examples to illustrate  the applicability of our results.

\section{weakly singular integral inequalities}
In this section, we first study the nonlinear integral
inequality (1.5). Then we discuss the weakly singular integral inequalities (1.6) and (1.7) by the integral
inequality (1.5).

\begin{thm}
Let $p\geq 1$, let $a(t), b(t)\in C[0,T)$ $(0<T\leq+\infty)$ be nonnegative, nondecreasing
functions, let $l(t)\in C(0,T)\bigcap L^{1}_{Loc}[0,T)$ be a nonnegative function, and let $\omega\in C[0,+\infty)$ be a nondecreasing, nonnegative function.
Assume that $u(t)$ is a
continuous and nonnegative function on $[0,T)$ with
\begin{equation} u(t)\leq
a(t)+b(t)(\int_{0}^{t}l(s)\omega(u(s))ds)^{\frac{1}{p}},
\qquad t \in [0,T).
\end{equation}
Then
\begin{equation}
\begin{split}
u(t)&\leq
2^{1-\frac{1}{p}}\left(\Omega^{-1}\left(\Omega(a^{p}(t))+b^{p}(t)\int_{0}^{t}l(s)ds\right)\right)^{\frac{1}{p}},
\qquad t\in[0,T_{1}],
\end{split}
\end{equation}
where $\Omega(x)=\int_{1}^{x}\frac{1}{\mu(t)}dt$,
$\mu(t)=\omega(2^{1-\frac{1}{p}}t^{\frac{1}{p}})$,
$\Omega^{-1}$ is the inverse of $\Omega$, and $T_{1}\in (0,T)$ is
such that $\Omega(a^{p}(t))+b^{p}(t)\int_{0}^{t}l(s)ds\in
\mathrm{Dom}(\Omega^{-1})$ for all $t\in [0,T_{1}]$.
\end{thm}
\begin{proof}Fix any $T_{0}\in[0,T_{1}]$. Then for $t\in[0,T_{0}]$, from (2.1), we have
\begin{equation}
u(t)\leq a(T_{0})+b(T_{0})(\int_{0}^{t}l(s)\omega(u(s))ds)^{\frac{1}{p}}.
\end{equation}
Let $\nu(t)=b^{p}(T_{0})\int_{0}^{t}l(s)\omega(u(s))ds$,
then we can get
\begin{equation}
\begin{split}
\nu^{'}(t)&= b^{p}(T_{0})l(t)\omega(u(t))\\
&\leq b^{p}(T_{0})l(t)\omega(a(T_{0})+\nu^{\frac{1}{p}}(t))\\
&\leq b^{p}(T_{0})l(t)\omega(2^{1-\frac{1}{p}}(a^{p}(T_{0})+\nu(t))^{\frac{1}{p}})\\
&= b^{p}(T_{0})l(t)\mu(a^{p}(T_{0})+\nu(t)).\\
\end{split}
\end{equation}
This yields
\begin{equation}
\frac{(a^{p}(T_{0}))^{'}+\nu^{'}(t)}{\mu(a^{p}(T_{0})+\nu(t))}=\frac{\nu^{'}(t)}{\mu(a^{p}(T_{0})+\nu(t))}\leq b^{p}(T_{0})l(t),\\
\end{equation}
or
\begin{equation}
\begin{split}
\frac{d}{dt}\Omega(a^{p}(T_{0})+\nu(t))&\leq b^{p}(T_{0})l(t).
\end{split}
\end{equation}
Integrating the inequality (2.6) from $0$ to $t\in[0,T_{0}]$, we obtain
\begin{equation}
\begin{split}
\Omega(a^{p}(T_{0})+\nu(t))&\leq \Omega(a^{p}(T_{0}))+\int_{0}^{t}b^{p}(T_{0})l(s)ds,
\end{split}
\end{equation}
then
\begin{equation}
\begin{split}
a^{p}(T_{0})+\nu(t)&\leq
\Omega^{-1}\left(\Omega(a^{p}(T_{0}))+b^{p}(T_{0})\int_{0}^{t}l(s)ds\right), \qquad
t\in[0,T_{0}].
\end{split}
\end{equation}
So
\begin{equation}
\begin{split}
a^{p}(T_{0})+\nu(T_{0})&\leq
\Omega^{-1}\left(\Omega(a^{p}(T_{0}))+b^{p}(T_{0})\int_{0}^{T_{0}}l(s)ds\right).
\end{split}
\end{equation}
Then, from (2.1) and (2.9), we know
\begin{equation}
\begin{split}
u(T_{0})&\leq a(T_{0})+\nu^{\frac{1}{p}}(T_{0})\\
&\leq 2^{1-\frac{1}{p}}(a^{p}(T_{0})+\nu(T_{0}))^{\frac{1}{p}}\\
&\leq 2^{1-\frac{1}{p}}\left(\Omega^{-1}\left(\Omega(a^{p}(T_{0}))+b^{p}(T_{0})\int_{0}^{T_{0}}l(s)ds\right)\right)^{\frac{1}{p}}.\\
\end{split}
\end{equation}
Now replace $T_{0}$ by $t$ in the inequality (2.10), we obtain the following
result
\begin{equation}
u(t)\leq 2^{1-\frac{1}{p}}\left(\Omega^{-1}\left(\Omega(a^{p}(t)+b^{p}(t)\int_{0}^{t}l(s)ds\right)\right)^{\frac{1}{p}},  \qquad
t\in[0,T_{1}],
\end{equation}
where $T_{1}\in(0,T)$ is such that
$\Omega(a^{p}(t))+b^{p}(t)\int_{0}^{t}l(s)ds\in\mathrm{Dom}(\Omega^{-1})$ for all $t\in
[0,T_{1}]$.
\end{proof}

As a consequence of Theorem 2.1, we can obtain the following result.
\begin{cor}
Let $p\geq 1$ and $0<\gamma\leq1$, let $a(t)$, $b(t)\in C[0,+\infty)$ be nondecreasing, nonnegative functions, let $l(t)\in C(0,+\infty)\bigcap L^{1}_{Loc}[0,+\infty)$ be a nonnegative
function, and $u(t)$ be a
continuous, nonnegative function on $[0,+\infty)$ with
\begin{equation} u(t)\leq
a(t)+b(t)(\int_{0}^{t}l(s)u^{p\gamma}(s)ds)^{\frac{1}{p}},
\qquad t \in [0,+\infty).
\end{equation}
Then the following assertions holds:\\
If $\gamma=1$, then we have
\begin{equation}
\begin{split}
u(t)&\leq
2^{1-\frac{1}{p}}a(t)\exp\left(\frac{2^{p-1}b^{p}(t)}{p}\int_{0}^{t}l(s)ds\right),
\qquad t\in[0,+\infty).
\end{split}
\end{equation}
If $0<\gamma<1$, then we get
\begin{equation}
\begin{split}
u(t)&\leq
2^{1-\frac{1}{p}}\left(a^{p(1-\gamma)}(t)+(1-\gamma)2^{(p-1)\gamma}b^{p}(t)\int_{0}^{t}l(s)ds\right)^{\frac{1}{p(1-\gamma)}},
\qquad t\in[0,+\infty).
\end{split}
\end{equation}
\end{cor}

Now, we begin to study the  weakly singular integral inequalities (1.6) and (1.7) by the above Theorem 2.1.
\begin{thm}
Let $\beta\in(0,1)$, $p>\frac{1}{\beta}$ and $q=\frac{p}{p-1}$, let $a(t), b(t)\in C[0,T)$ be nondecreasing, nonnegative functions, let $l(t)\in C(0,T)\bigcap L^{p}_{Loc}[0,T)$ be a nonnegative
function, let $\omega\in C[0,+\infty)$ be a nondecreasing, nonnegative function, Assume that $u(t)$ is a
continuous, nonnegative function on $[0,T)$ with
\begin{equation}
u(t)\leq
a(t)+b(t)\int_{0}^{t}(t-s)^{\beta-1}l(s)\omega(u(s))ds,
\qquad t\in[0,T).
\end{equation}
Then
\begin{equation}
u(t)\leq 2^{1-\frac{1}{p}}\left(\Omega^{-1}\left(\Omega(a^{p}(t))+c^{p}(t)\int_{0}^{t}l^{p}(s)ds\right)\right)^{\frac{1}{p}},  \qquad
t\in[0,T_{1}],
\end{equation}
where $c(t)=\frac{t^{\beta-1+\frac{1}{q}}b(t)}{(q(\beta-1)+1)^{\frac{1}{q}}}$, $\Omega(x)=\int_{1}^{x}\frac{1}{\mu(t)}dt$,
$\mu(t)=\omega^{p}(2^{1-\frac{1}{p}}t^{\frac{1}{p}})$,
$\Omega^{-1}$ is the inverse of $\Omega$, and $T_{1}\in(0,T)$ is such that
$\Omega(a^{p}(t))+c^{p}(t)\int_{0}^{t}l^{p}(s)ds\in\mathrm{Dom}(\Omega^{-1})$ for all $t\in
[0,T_{1}]$.
\end{thm}
\begin{proof}From (2.15), since $p>\frac{1}{\beta}$ and $\frac{1}{p}+\frac{1}{q}=1$, then $q(\beta-1)+1>0$, using the
H\"{o}lder inequality, we know
\begin{equation}
\begin{split}
u(t)&\leq a(t)+b(t)(\int_{0}^{t}(t-s)^{q(\beta-1)}ds)^{\frac{1}{q}}(\int_{0}^{t}(l(s)\omega(u(s)))^{p}ds)^{\frac{1}{p}}\\
&\leq a(t)+\frac{t^{\beta-1+\frac{1}{q}}b(t)}{(q(\beta-1)+1)^{\frac{1}{q}}}(\int_{0}^{t}l^{p}(s)\omega^{p}(u(s))ds)^{\frac{1}{p}}.\\
\end{split}
\end{equation}
Then, by Theorem 2.1, we get
\begin{equation}
u(t)\leq 2^{1-\frac{1}{p}}\left(\Omega^{-1}\left(\Omega(a^{p}(t))+c^{p}(t)\int_{0}^{t}l^{p}(s)ds\right)\right)^{\frac{1}{p}},  \qquad
t\in[0,T_{1}],
\end{equation}
where $T_{1}\in(0,T)$ is such that
$\Omega(a^{p}(t))+c^{p}(t)\int_{0}^{t}l^{p}(s)ds\in\mathrm{Dom}(\Omega^{-1})$ for all $t\in
[0,T_{1}]$.
Thus, we complete the proof of Theorem 2.3.
\end{proof}

\begin{thm}
Let $a, b\geq 0$, $\alpha>\delta\geq 0$ and $0<\beta<1$, $p>\frac{1}{\beta}$ and $q=\frac{p}{p-1}$, $f:(0,T)\times\mathbb{R}\rightarrow\mathbb{R}$ be a continuous function and
$$|f(t,x)|\leq l(t)\omega(t^{\alpha}|x|),$$
where nonnegative
function $l(t)\in C(0,T)\bigcap L^{p}_{Loc}[0,T)$ and nondecreasing, nonnegative function $\omega\in C[0,+\infty)$, and let $t^{\alpha}u(t)$ be a
continuous, nonnegative function on $[0,T)$ with
\begin{equation}
u(t)\leq
at^{-\alpha}+bt^{-\delta}\int_{0}^{t}(t-s)^{\beta-1}f(s,u(s))ds,
\qquad t\in(0,T).
\end{equation}
Then
\begin{equation}
u(t)\leq 2^{1-\frac{1}{p}}t^{-\alpha}\left(\Omega^{-1}\left(\Omega(a^{p})+c^{p}(t)\int_{0}^{t}l^{p}(s)ds\right)\right)^{\frac{1}{p}},  \qquad
t\in(0,T_{1}],
\end{equation}
where $c(t)=\frac{bt^{\alpha-\delta+\beta-1+\frac{1}{q}}}{(q(\beta-1)+1)^{\frac{1}{q}}}$, $\Omega(x)=\int_{1}^{x}\frac{1}{\mu(t)}dt$,
$\mu(t)=\omega^{p}(2^{1-\frac{1}{p}}t^{\frac{1}{p}})$,
$\Omega^{-1}$ is the inverse of $\Omega$, and $T_{1}\in(0,T)$ is such that
$\Omega(a^{p})+c^{p}(t)\int_{0}^{t}l^{p}(s)ds\in\mathrm{Dom}(\Omega^{-1})$ for all $t\in
[0,T_{1}]$.
\end{thm}
\begin{proof}
Let $v(t)=t^{\alpha}u(t)$, then we get
\begin{equation}
\begin{split}
v(t)&\leq
a+bt^{\alpha-\delta}\int_{0}^{t}(t-s)^{\beta-1}f(s,s^{-\alpha}v(s))ds\\
&\leq
a+bt^{\alpha-\delta}\int_{0}^{t}(t-s)^{\beta-1}l(s)\omega(v(s))ds\\
&\leq
a+\frac{bt^{\alpha-\delta+\beta-1+\frac{1}{q}}}{(q(\beta-1)+1)^{\frac{1}{q}}}(\int_{0}^{t}l^{p}(s)\omega^{p}(u(s))ds)^{\frac{1}{p}}.\\
\end{split}
\end{equation}
Using Theorem 2.1, we obtain the inequality (2.20) and complete the proof.
\end{proof}

In Theorem 2.4, if $\alpha=1-\beta$ and $\delta=0$, then we obtain the following conclusion under the weaker condition $t^{1-\beta}l(t)\in C(0,T)\bigcap L^{p}_{Loc}[0,T)$.
\begin{thm}
Let $a\geq 0$ and $0<\beta<1$, $p>\frac{1}{\beta}$ and $q=\frac{p}{p-1}$, let $b(t)$ be a nondecreasing, nonnegative continuous function on $[0,T)$, $f:(0,T)\times\mathbb{R}\rightarrow\mathbb{R}$ be a continuous function and
$$|f(t,x)|\leq l(t)\omega(t^{1-\beta}|x|),$$
where nonnegative
function $t^{1-\beta}l(t)\in C(0,T)\bigcap L^{p}_{Loc}[0,T)$ and nondecreasing, nonnegative function $\omega\in C[0,+\infty)$, and let $t^{1-\beta}u(t)$ be a
continuous, nonnegative function on $[0,T)$ with
\begin{equation}
u(t)\leq
at^{\beta-1}+b(t)\int_{0}^{t}(t-s)^{\beta-1}f(s,u(s))ds,
\qquad t\in(0,T).
\end{equation}
Then
\begin{equation}
u(t)\leq
2^{1-\frac{1}{p}}t^{\beta-1}\left(\Omega^{-1}\left(\Omega(a^{p})+c^{p}(t)\int_{0}^{t}s^{p(1-\beta)}l^{p}(s)ds\right)\right)^{\frac{1}{p}}, \qquad t\in(0,T_{1}],
\end{equation}
where $c(t)=\frac{2^{\frac{1}{q}}b(t)t^{\beta-1+\frac{1}{q}}}{(q\beta-q+1)^{\frac{1}{q}}}$, $\Omega(x)=\int_{1}^{x}\frac{1}{\mu(t)}dt$,
$\mu(t)=\omega^{p}(2^{1-\frac{1}{p}}t^{\frac{1}{p}})$,
$\Omega^{-1}$ is the inverse of $\Omega$, and $T_{1}\in (0,T)$ is
such that $\Omega(a^{p})+c^{p}(t)\int_{0}^{t}s^{p(1-\beta)}l^{p}(s)ds\in
\mathrm{Dom}(\Omega^{-1})$ for all $t\in [0,T_{1}]$.
\end{thm}
\begin{proof}
Let $v(t)=t^{1-\beta}u(t)$, from (2.22) and $0<q(1-\beta)<1$, using the
H\"{o}lder inequality, we know
\begin{equation}
\begin{split}
v(t)&\leq
a+t^{1-\beta}b(t)\int_{0}^{t}(t-s)^{\beta-1}f(s,s^{\beta-1}v(s))ds\\
&\leq a+t^{1-\beta}b(t)\int_{0}^{t}(t-s)^{\beta-1}l(s)\omega(v(s))ds\\
&= a+b(t)\int_{0}^{t}(\frac{1}{t-s}+\frac{1}{s})^{1-\beta}s^{1-\beta}l(s)\omega(v(s))ds\\
&\leq a+b(t)(\int_{0}^{t}(\frac{1}{t-s}+\frac{1}{s})^{q(1-\beta)}ds)^{\frac{1}{q}}(\int_{0}^{t}s^{p(1-\beta)}l^{p}(s)\omega^{p}(v(s))ds)^{\frac{1}{p}}\\
&\leq a+b(t)(\int_{0}^{t}(t-s)^{q(\beta-1)}+s^{q(\beta-1)}ds)^{\frac{1}{q}}(\int_{0}^{t}s^{p(1-\beta)}l^{p}(s)\omega^{p}(v(s))ds)^{\frac{1}{p}}\\
&= a+\frac{2^{\frac{1}{q}}b(t)t^{\beta-1+\frac{1}{q}}}{(q\beta-q+1)^{\frac{1}{q}}}(\int_{0}^{t}s^{p(1-\beta)}l^{p}(s)\omega^{p}(v(s))ds)^{\frac{1}{p}}.\\
\end{split}
\end{equation}
By Theorem 2.1, then we have
\begin{equation}
v(t)\leq
2^{1-\frac{1}{p}}\left(\Omega^{-1}\left(\Omega(a^{p})+c^{p}(t)\int_{0}^{t}s^{p(1-\beta)}l^{p}(s)ds\right)\right)^{\frac{1}{p}}, \qquad t\in[0,T_{1}].
\end{equation}
Thus, we complete the proof.
\end{proof}

As a consequence of Theorem 2.5, we can get the following conclusion.
\begin{cor}
Let $a\geq 0$, $0<\gamma\leq1$ and $0<\beta<1$, $p>\frac{1}{\beta}$ and $q=\frac{p}{p-1}$, let $b(t)$ be a nondecreasing, nonnegative continuous function on $[0,+\infty)$, let $l(t)$ be a nonnegative
function with $t^{(1-\gamma)(1-\beta)}l(t)\in C(0,+\infty)\bigcap L^{p}_{Loc}[0,+\infty)$, and let $t^{1-\beta}u(t)$ be a
continuous, nonnegative function on $[0,+\infty)$ with
\begin{equation}
u(t)\leq
at^{\beta-1}+b(t)\int_{0}^{t}(t-s)^{\beta-1}l(s)u^{\gamma}(s)ds,
\qquad t\in(0,+\infty).
\end{equation}
Then the following assertions holds:\\
If $\gamma=1$, we can obtain
\begin{equation}
u(t)\leq
2^{1-\frac{1}{p}}at^{\beta-1}\exp\left(\frac{2^{p-1}c^{p}(t)}{p}\int_{0}^{t}l^{p}(s)ds\right), \qquad t\in(0,+\infty),
\end{equation}
where $c(t)=\frac{2^{\frac{1}{q}}b(t)t^{\beta-1+\frac{1}{q}}}{(q\beta-q+1)^{\frac{1}{q}}}$.\\
If $0<\gamma<1$, then we can obtain
\begin{equation}
u(t)\leq
2^{1-\frac{1}{p}}t^{\beta-1}\left(a^{p(1-\gamma)}+(1-\gamma)2^{(p-1)\gamma}c^{p}(t)\int_{0}^{t}s^{p(1-\gamma)(1-\beta)}l^{p}(s)ds\right)^{\frac{1}{p(1-\gamma)}}, \qquad t\in(0,+\infty),
\end{equation}
where $c(t)$ is defined as in (2.27).
\end{cor}

\begin{rem}
Willett [16] and Zhu [20] studied the inequality (2.1) for the special case $\omega(t)=t^{p}$. Medve\v{d} [8,Theorem 1], Medve\v{d} [9,Theorem 2] and Zhu [19, Theorem 2.7] studied the nonlinear integral inequality (2.15) under the assumptions that $a(t)\in C^{1}[0,T)$  is a nondecreasing, nonnegative function and $b(t)$ is a constant. In Theorem 2.3, we only suppose that $a(t), b(t)$ are nondecreasing, nonnegative continuous functions on $[0,T)$. Hence, our results improve and generalize Theorem 1 in [8],  Theorem 2 in [9] and Theorem 2.7 in [19]. By a reduction to the classical Gronwall inequality, Webb [15] studied the inequality (2.19) when $f(t,u(t))=t^{-\gamma}u(t)$. In [20], Zhu proved the inequality (2.19) when $f(t,u(t))=l(t)u(t)$. In Theorem 2.4, we study the inequality (2.19) when $|f(t,x)|\leq l(t)\omega(t^{\alpha}|x|)$, where $l(t)\in C(0,T)\bigcap L^{p}_{Loc}[0,T)$. For the special $\alpha=1-\beta$ and $\delta=0$, we study the inequality (2.22) under the weaker condition $t^{1-\beta}l(t)\in C(0,T)\bigcap L^{p}_{Loc}[0,T)$. Theorem 2.4, Theorem 2.5 and Corollary 2.6 generalize Theorem 3.9 in [15] and Theorem 3.4 in [20].
\end{rem}

\begin{exa}
Suppose that $t^{\frac{1}{2}}u(t)$ is a
continuous, nonnegative function and satisfies the following inequality
\begin{equation}
u(t)\leq
t^{\frac{-1}{2}}+t^{\frac{-1}{3}}\int_{0}^{t}(t-s)^{-\frac{1}{3}}s^{\frac{-1}{12}}u^{\frac{1}{2}}(s)ds.
\end{equation}
\end{exa}
Let $p=2$. Using Theorem 2.4, we know $l(t)=t^{\frac{-1}{3}}$, $\mu(t)=2^{\frac{1}{2}}t^{\frac{1}{2}}$ and $\Omega(x)=2^{\frac{1}{2}}(x^{\frac{1}{2}}-1)$, then we obtain
\begin{equation}
u(t)\leq2^{\frac{1}{2}}t^{\frac{-1}{2}}+9t^{\frac{1}{2}},  \qquad
t\in(0,+\infty).
\end{equation}

\begin{exa}
Suppose that $t^{\frac{1}{3}}u(t)$ is a
continuous, nonnegative function and satisfies the inequality
\begin{equation}
u(t)\leq
t^{-\frac{1}{3}}+\int_{0}^{t}(t-s)^{-\frac{1}{3}}s^{\frac{-1}{2}}u^{\frac{1}{2}}(s)ds.
\end{equation}
\end{exa}
Let $p=2$. By Corollary 2.6, we get
\begin{equation}
u(t)\leq 2^{\frac{1}{2}}t^{\frac{-1}{3}}+18t^{\frac{1}{3}},  \qquad
t\in(0,+\infty).
\end{equation}
\begin{rem}
From (2.31), we know $t^{\frac{-1}{2}}u^{\frac{1}{2}}(t)=t^{\frac{-2}{3}}(t^{\frac{1}{3}}u(t))^{\frac{1}{2}}$ and $t^{\frac{-2}{3}}\notin L^{p}_{Loc}[0,T)$$(p>\frac{3}{2})$, hence we can not obtain the estimate of the inequality (2.31) by Theorem 2.4. From Example 2.8 and Example 2.9,
we know that Theorem 2.5 is better than Theorem 2.4 for the special $\alpha=1-\beta$ and $\delta=0$.
\end{rem}

\section{global solutions for fractional differential equation}
In this section, we give the existence and uniqueness results for the fractional differential equation (1.8). we first introduce some notations, definitions, and preliminary facts which are used throughout this section. For more details about fractional calculus, we refer the reader to the
texts ([3,7,10,12]).

Let $\alpha\in (0,1)$, we denote $C_{\alpha}(0,T]=\{x(t):
x(t)\in C(0,T]$ and $\lim_{t\rightarrow 0^{+}}t^{\alpha}x(t)$
exists $\}$. Let $\|x\|_{\alpha}=\sup_{0<t\leq T}t^{\alpha}|x(t)|$, then $C_{\alpha}(0,T]$ endowed with the norm $\|\cdot\|_{\alpha}$ is a Banach space. We denote $C_{\alpha}(0,+\infty)=\{x(t):
x(t)\in C(0,+\infty)$ and $\lim_{t\rightarrow 0^{+}}t^{\alpha}x(t)$
exists $\}$.

\begin{defn}\label{defn: 3.1}[3]
Let $\beta\in(0,1)$, The operator $I^{\beta}$, defined on $L^{1}[0,T]$ by
$$I^{\beta}\varphi(t)=\frac{1}{\Gamma(\beta)}\int_{0}^{t}\frac{\varphi(s)}{(t-s)^{1-\beta}}ds, \qquad t\in[0,T], $$
is called the Riemann-Liouville fractional integral operator of
order $\beta$.
\end{defn}

\begin{defn}\label{defn: 3.2}[3]
Let $\beta\in(0,1)$, The operator $D_{r}^{\beta}$, defined by
$$D_{r}^{\beta}\varphi(t)=\frac{d}{dt}I^{1-\beta}\varphi(t)=\frac{1}{\Gamma(1-\beta)}\frac{d}{dt}\int_{0}^{t}\frac{\varphi(s)}{(t-s)^{\beta}}ds,\qquad a.e.t\in[0,T],$$
where $I^{1-\beta}\varphi(t)$ is an absolutely continuous function, is called the Riemann-Liouville fractional differential operator of order $\beta$.
\end{defn}

\begin{thm}\label{thm: 3.3}[13]
Let $S$ be a convex subset of a Banach space $E$ and assume $0\in S$. Let $F: S\rightarrow S$ be a continuous and compact map, and let the set $\{s\in S :x=\lambda Fx$ for some $\lambda\in(0,1)\}$ be bounded. Then $F$ has at least one fixed point in $S$.
\end{thm}

\begin{lem}
Let $0<\mu<1$, then
$$\varphi(t)=\frac{t^{\mu}-1}{(t-1)^{\mu}}$$
is a nondecreasing function on $(1,+\infty)$ and $\lim_{t\rightarrow 1^{+}}\varphi(t)=0$.
\end{lem}
\begin{proof}
We know
$$\varphi{'}(t)=\frac{\mu(t-1)^{\mu-1}[1-t^{\mu-1}]}{(t-1)^{2\mu}},$$
then $\varphi{'}(t)>0$ for $t\in(1,+\infty)$. Hence $\varphi(t)$ is a nondecreasing function on $(1,+\infty)$. Next, we get
$$\lim_{t\rightarrow 1^{+}}\varphi(t)=\lim_{t\rightarrow 1^{+}}\frac{(t-1)^{1-\mu}}{t^{1-\mu}}=0.$$
Thus, we complete the proof.
\end{proof}

\begin{lem}
Suppose that $s^{1-\beta}\rho(s)\in L^{p}[0,1]$, then
\begin{equation}
\left|\int_{0}^{t}(\frac{t}{t-s})^{1-\beta}\rho(s)ds\right|\leq \frac{2^{\frac{1}{q}}t^{\beta-1+\frac{1}{q}}}{(q\beta-q+1)^{\frac{1}{q}}}\left(\int_{0}^{t}s^{p(1-\beta)}|\rho(s)|^{p}ds\right)^{\frac{1}{p}}
\end{equation}
for $t\in[0,1]$, and if $0<t_{1}\leq t_{2}\leq1$, then
\begin{equation}
\begin{split}
&\left|\int_{0}^{t_{2}}(\frac{t_{2}}{t_{2}-s})^{1-\beta}\rho(s)ds-\int_{0}^{t_{1}}(\frac{t_{1}}{t_{1}-s})^{1-\beta}\rho(s)ds\right|\\
&\leq\frac{2^{\frac{1}{q}}(t_{2}-t_{1})^{\beta-1+\frac{1}{q}}}{(q\beta-q+1)^{\frac{1}{q}}}\left(\int_{t_{1}}^{t_{2}}s^{p(1-\beta)}|\rho(s)|^{p}ds\right)^{\frac{1}{p}}\\
&+\left(\frac{(t_{2}-t_{1})^{1+q(\beta-1)}+t_{1}^{1+q(\beta-1)}-t_{2}^{1+q(\beta-1)}}{q\beta-q+1}\right)^{\frac{1}{q}}\left(\int_{0}^{t_{1}}s^{p(1-\beta)}|\rho(s)|^{p}ds\right)^{\frac{1}{p}}.\\
\end{split}
\end{equation}
where $\beta\in(0,1)$, $p>\frac{1}{\beta}$ and $q=\frac{p}{p-1}$.
\end{lem}
\begin{proof}
By the H\"{o}lder inequality and $0<q(1-\beta)<1$, we have
\begin{equation}
\begin{split}
\left|\int_{0}^{t}(\frac{t}{t-s})^{1-\beta}\rho(s)ds\right|&\leq
\int_{0}^{t}(\frac{1}{t-s}+\frac{1}{s})^{1-\beta}s^{1-\beta}|\rho(s)|ds\\
&\leq \left(\int_{0}^{t}(\frac{1}{t-s}+\frac{1}{s})^{q(1-\beta)}ds\right)^{\frac{1}{q}}\left(\int_{0}^{t}s^{p(1-\beta)}|\rho(s)|^{p}ds\right)^{\frac{1}{p}}\\
&\leq \left(\int_{0}^{t}(t-s)^{q(\beta-1)}+s^{q(\beta-1)}ds\right)^{\frac{1}{q}}\left(\int_{0}^{t}s^{p(1-\beta)}|\rho(s)|^{p}ds\right)^{\frac{1}{p}}\\
&=\frac{2^{\frac{1}{q}}t^{\beta-1+\frac{1}{q}}}{(q\beta-q+1)^{\frac{1}{q}}}\left(\int_{0}^{t}s^{p(1-\beta)}|\rho(s)|^{p}ds\right)^{\frac{1}{p}}.\\
\end{split}
\end{equation}
Thus, we obtain the inequality (3.1).

Now, let us prove the inequality (3.2), from (3.2), we get
\begin{eqnarray*}
\lefteqn{\left|\int_{0}^{t_{2}}(\frac{t_{2}}{t_{2}-s})^{1-\beta}\rho(s)ds-\int_{0}^{t_{1}}(\frac{t_{1}}{t_{1}-s})^{1-\beta}\rho(s)ds\right|}\\
&\leq&\left|\int_{t_{1}}^{t_{2}}(\frac{t_{2}}{t_{2}-s})^{1-\beta}\rho(s)ds\right|+\left|\int_{0}^{t_{1}}\left[(\frac{t_{2}}{t_{2}-s})^{1-\beta}-(\frac{t_{1}}{t_{1}-s})^{1-\beta}\right]\rho(s)ds\right|.\\
\end{eqnarray*}
Similarly with the proof of the inequality (3.3), we obtain
\begin{equation}
\left|\int_{t_{1}}^{t_{2}}(\frac{t_{2}}{t_{2}-s})^{1-\beta}\rho(s)ds\right|\leq\frac{2^{\frac{1}{q}}(t_{2}-t_{1})^{\beta-1+\frac{1}{q}}}{(q\beta-q+1)^{\frac{1}{q}}}\left(\int_{t_{1}}^{t_{2}}s^{p(1-\beta)}|\rho(s)|^{p}ds\right)^{\frac{1}{p}}.\\
\end{equation}
From Lemma 3.4, we get
\begin{equation}
(\frac{x}{x-1})^{1-\beta}-(\frac{1}{x-1})^{1-\beta}\leq(\frac{y}{y-1})^{1-\beta}-(\frac{1}{y-1})^{1-\beta}
\end{equation}
for $1<x\leq y<+\infty$. Let $x=\frac{t_{1}}{s}$ and $y=\frac{t_{2}}{s}$ in (3.5), then
\begin{equation}
(\frac{t_{1}}{t_{1}-s})^{1-\beta}-(\frac{s}{t_{1}-s})^{1-\beta}\leq(\frac{t_{2}}{t_{2}-s})^{1-\beta}-(\frac{s}{t_{2}-s})^{1-\beta},
\end{equation}
where $0<s< t_{1}\leq t_{2}\leq1$. Then we know
\begin{equation}
\begin{split}
&\int_{0}^{t_{1}}\left|(\frac{t_{2}}{t_{2}-s})^{1-\beta}-(\frac{t_{1}}{t_{1}-s})^{1-\beta}\right||\rho(s)|ds
\leq\int_{0}^{t_{1}}\left[(\frac{s}{t_{1}-s})^{1-\beta}-(\frac{s}{t_{2}-s})^{1-\beta}\right]|\rho(s)|ds\\
&\leq\left(\int_{0}^{t_{1}}\left[(\frac{1}{t_{1}-s})^{1-\beta}-(\frac{1}{t_{2}-s})^{1-\beta}\right]^{q}ds\right)^{\frac{1}{q}}\left(\int_{0}^{t_{1}}s^{p(1-\beta)}|\rho(s)|^{p}ds\right)^{\frac{1}{p}}\\
&\leq\left(\int_{0}^{t_{1}}(\frac{1}{t_{1}-s})^{q(1-\beta)}-(\frac{1}{t_{2}-s})^{q(1-\beta)}ds\right)^{\frac{1}{q}}\left(\int_{0}^{t_{1}}s^{p(1-\beta)}|\rho(s)|^{p}ds\right)^{\frac{1}{p}}\\
&\leq\left(\frac{(t_{2}-t_{1})^{1+q(\beta-1)}+t_{1}^{1+q(\beta-1)}-t_{2}^{1+q(\beta-1)}}{q\beta-q+1}\right)^{\frac{1}{q}}\left(\int_{0}^{t_{1}}s^{p(1-\beta)}|\rho(s)|^{p}ds\right)^{\frac{1}{p}}.\\
\end{split}
\end{equation}
By the inequality (3.4) and (3.7), then the inequality (3.2) holds for all $0<t_{1}\leq t_{2}\leq1$.
\end{proof}

\begin{thm}
Let $p>\frac{1}{\beta}$ and $q=\frac{p}{p-1}$. Suppose  $f: (0,T]\times\mathbb{R}\rightarrow\mathbb{R}$ is a continuous
function, and there exist nonnegative function $l(t)$ with $t^{1-\beta}l(t)\in
C(0,T]\bigcap L^{p}[0,T]$ and nonnegative
nondecreasing function $\omega\in C[0,+\infty)$ such that
$$|f(t, x)|\leq l(t)\omega(t^{1-\beta}|x|)$$
for all $(t,x)\in (0,T]\times\mathbb{R}$. Then the fractional differential equation
(1.8) has at least one solution in $C_{1-\beta}(0,T]$ provided that
$$\Omega(|x_{0}|^{p})+c^{p}(t)\int_{0}^{t}s^{p(1-\beta)}l^{p}(s)ds\in\mathrm{Dom}(\Omega^{-1})$$
for all $t\in[0,T]$, where $c(t)=\frac{2^{\frac{1}{q}}t^{\beta-1+\frac{1}{q}}}{\Gamma(\beta)(q(\beta-1)+1)^{\frac{1}{q}}}$, $\Omega(x)=\int_{1}^{x}\frac{1}{\mu(t)}dt$,
$\mu(t)=\omega^{p}(2^{1-\frac{1}{p}}t^{\frac{1}{p}})$,
$\Omega^{-1}$ is the inverse of $\Omega$.
\end{thm}
\begin{proof}
Step 1. Since $|f(t, x(t))|\leq l(t)\omega(t^{1-\beta}|x(t)|)$, $x(t)\in C_{1-\beta}(0,T]$ and $t^{1-\beta}l(t)\in
C(0,T]\bigcap L^{p}[0,T]$, then we know $f(t, x(t))\in L^{1}[0,T]$. By Theorem 6.2 of [1], we know that $x(t)\in C_{1-\beta}(0,T]$ satisfies the fractional differential equation (1.8) if and only if it satisfies the following Volterra integral equation
\begin{equation}
x(t)=x_{0}t^{\beta-1}+\frac{1}{\Gamma(\beta)}\int_{0}^{t}(t-s)^{\beta-1}f(s,x(s))ds,\qquad
t\in(0, T].
\end{equation}

Step 2. Let $G: C_{1-\beta}(0,T]\rightarrow C_{1-\beta}(0,T]$ be the
operator defined by
\begin{equation}
(Gx)(t)=x_{0}t^{\beta-1}+\frac{1}{\Gamma(\beta)}\int_{0}^{t}(t-s)^{\beta-1}f(s,x(s))ds.
\end{equation}
We will prove that $G$ is a compact operator. To see this, let $\Omega\in C_{1-\beta}(0,T]$ be bounded and $\|x\|_{1-\beta}\leq M$ for each $x\in \Omega$,
we will show that $t^{1-\beta}G(\Omega)$ is uniformly bounded and
equicontinuous on $[0,T]$. Firstly, we prove that $t^{1-\beta}G(\Omega)$ is uniformly bounded. Let $x\in \Omega$, then we have
\begin{equation}
\begin{split}
|t^{1-\beta}Gx(t)|
&\leq|x_{0}|+\frac{t^{1-\beta}}{\Gamma(\beta)}|\int_{0}^{t}(t-s)^{\beta-1}f(s,x(s))ds|\\
&\leq |x_{0}|+\frac{1}{\Gamma(\beta)}(\int_{0}^{t}(\frac{1}{t-s}+\frac{1}{s})^{q(1-\beta)}ds)^{\frac{1}{q}}(\int_{0}^{t}s^{p(1-\beta)}l^{p}(s)\omega^{p}(s^{1-\beta}|x(s)|)ds)^{\frac{1}{p}}\\
&\leq |x_{0}|+\frac{2^{\frac{1}{q}}\omega(M)t^{\beta-1+\frac{1}{q}}}{\Gamma(\beta)(q(\beta-1)+1)^{\frac{1}{q}}}(\int_{0}^{t}s^{p(1-\beta)}l^{p}(s)ds)^{\frac{1}{p}}\\
&\leq |x_{0}|+\frac{2^{\frac{1}{q}}\omega(M)T^{\beta-1+\frac{1}{q}}}{\Gamma(\beta)(q(\beta-1)+1)^{\frac{1}{q}}}(\int_{0}^{T}s^{p(1-\beta)}l^{p}(s)ds)^{\frac{1}{p}}.\\
\end{split}
\end{equation}
This proves that the set that $t^{1-\beta}G(\Omega)$ is uniformly bounded. Secondly, we prove that $t^{1-\beta}G(\Omega)$ is an equicontinuous family. For any $x\in\Omega$, let $0\leq t_{1}<t_{2}\leq T$,
we get
\begin{equation}
\begin{split}
|t_{2}^{1-\beta}Gx(t_{2})-t_{1}^{1-\beta}Gx(t_{1})|
&=\frac{1}{\Gamma(\beta)}|\int_{0}^{t_{2}}(\frac{t_{2}}{t_{2}-s})^{1-\beta}f(s,x(s))ds-\int_{0}^{t_{1}}(\frac{t_{1}}{t_{1}-s})^{1-\beta}f(s,x(s))ds|.\\
\end{split}
\end{equation}
Since $|f(t, x(t))|\leq l(t)\omega(t^{1-\beta}|x(t)|)\leq l(t)\omega(M)$ and $t^{1-\beta}l(t)\in
C(0,T]\bigcap L^{p}[0,T]$. By Lemma 3.5, we know that the right-hand side of the above
inequality (3.11) tends to zero as $t_{2}\rightarrow t_{1}$. Therefore, $t^{1-\beta}G(\Omega)$ is an equicontinuous family.

Step 3. We now show that $G$ is continuous, that is $x_{n}\rightarrow x$ (that is $x_{n}(t)\rightarrow x(t)$ uniformly in $t$) implies $Gx_{n}\rightarrow Gx$. Since $x_{n}\rightarrow x$ in $C_{1-\beta}(0,T]$, then there exists
$r>0$ such that $\|x_{n}\|_{1-\beta}\leq r$ and $\|x\|_{1-\beta}\leq r$.
For every $s\in(0, T]$, we have
$$f(s,x_{n}(s))\rightarrow f(s,x(s)).$$
Using the assumption of $f$, we get
\begin{equation}
\begin{split}
(\frac{t}{t-s})^{1-\beta}|f(s,x_{n}(s))-f(s,x(s))|
&\leq
2\omega(r)(\frac{t}{t-s})^{1-\beta}l(s).\\
\end{split}
\end{equation}
Since $t^{1-\beta}l(t)\in C(0,T]\bigcap L^{p}[0,T]$, using (3.1) in Lemma 3.5, then we know that the function $$s\rightarrow
2\omega(r)(\frac{t}{t-s})^{1-\beta}l(s)$$ is integrable for $s\in
(0,t)$. By means of the Lebesgue dominated convergence theorem yields
$$|\int_{0}^{t}(\frac{t}{t-s})^{1-\beta}[f(s, x_{n}(s))-f(s, x(s))]ds|\rightarrow 0$$
as $n\rightarrow +\infty$. Therefore $t^{1-\beta}Gx_{n}(t)\rightarrow t^{1-\beta}Gx(t)$
pointwise on $[0,T]$ as $n\rightarrow +\infty$. With the fact that $G$ is compact, we get $G: C_{1-\beta}(0,T]\rightarrow C_{1-\beta}(0,T]$ is continuous.

Step 4. We shall prove that the set $\Lambda=\{x \in
C_{1-\beta}(0,T]: x=\lambda Gx$ for some $0<\lambda<1\}$ is
bounded. Indeed, let $x\in\Lambda$, then
\begin{equation}
\begin{split}
|x(t)|&\leq
|x_{0}|t^{\beta-1}+\frac{1}{\Gamma(\beta)}\int_{0}^{t}(t-s)^{\beta-1}|f(s,x(s))|ds\\
&\leq
|x_{0}|t^{\beta-1}+\frac{1}{\Gamma(\beta)}\int_{0}^{t}(t-s)^{\beta-1}l(s)\omega(s^{1-\beta}|x(s)|)ds.\\
\end{split}
\end{equation}
By Theorem 2.5, we obtain
\begin{equation}
x(t)\leq 2^{1-\frac{1}{p}}t^{\beta-1}\left(\Omega^{-1}\left(\Omega(|x_{0}|^{p})+c^{p}(t)\int_{0}^{t}s^{p(1-\beta)}l^{p}(s)ds\right)\right)^{\frac{1}{p}},  \qquad
t\in[0,T],
\end{equation}
and
\begin{equation}
\|x\|_{1-\beta}\leq 2^{1-\frac{1}{p}}\left(\Omega^{-1}\left(\Omega(|x_{0}|^{p})+c^{p}(T)\int_{0}^{T}s^{p(1-\beta)}l^{p}(s)ds\right)\right)^{\frac{1}{p}}.
\end{equation}
Then, the set $\Lambda$ is bounded.

Finally, applying fixed point theorem 3.3, the operator
$G$ has a fixed point $x(t)\in C_{1-\beta}(0,T]$ which is the
solution of the equation (1.8).
\end{proof}

Now, we investigate the existence of global solutions of the fractional differential
equation (1.8).
\begin{thm}
Let $p>\frac{1}{\beta}$ and $q=\frac{p}{p-1}$. Suppose  $f: (0,+\infty)\times\mathbb{R}\rightarrow\mathbb{R}$ is a continuous
function, and there exist nonnegative function $l(t)\in C(0,+\infty)$ with $t^{1-\beta}l(t)\in L^{p}_{Loc}[0,+\infty)$ and nonnegative
nondecreasing function $\omega\in C[0,+\infty)$ with $\lim_{t\rightarrow+\infty}\frac{t}{\omega(t)}=K (0<K\leq+\infty)$ such that
$$|f(t, x)|\leq l(t)\omega(t^{1-\beta}|x|)$$
for all $(t,x)\in (0,+\infty)\times\mathbb{R}$. Then the fractional differential equation
(1.8) has at least one global solution in $C_{1-\beta}(0,+\infty)$.
\end{thm}
\begin{proof}
We know
\begin{equation}
\lim_{t\rightarrow+\infty}\frac{t}{\mu(t)}=\lim_{t\rightarrow+\infty}\frac{t}{\omega^{p}(2^{1-\frac{1}{p}}t^{\frac{1}{p}})}=2^{1-p}K^{p}.
\end{equation}
Since $\int_{0}^{+\infty}\frac{1}{t}dt$ is divergent, from (3.16), we get that $\int_{0}^{+\infty}\frac{1}{\mu(t)}dt$ is also divergent. Then we get $[0,+\infty)\in\mathrm{Dom}(\Omega^{-1})$ and $\Omega(|x_{0}|^{p})+c^{p}(t)\int_{0}^{t}s^{p(1-\beta)}l^{p}(s)ds\in\mathrm{Dom}(\Omega^{-1})$
for every $t\in[0,+\infty)$, where $\Omega(x)$ and $c(t)$ are defined as in Theorem 3.6.

For any $T>0$, from Theorem 3.6, We know that the equation (1.8) has at least one
solution in $C_{1-\beta}(0, T]$. Since $T$ can be chosen arbitrarily large, then the equation (1.8)
has at least one global solution in $C_{1-\beta}(0,+\infty)$. Thus, we complete the proof of Theorem 3.7.
\end{proof}

From Theorem 3.7, we can immediately obtain the following conclusion.
\begin{cor}
Let $0<\gamma\leq1$, $p>\frac{1}{\beta}$ and $q=\frac{p}{p-1}$. Suppose  $f: (0,+\infty)\times\mathbb{R}\rightarrow\mathbb{R}$ is a continuous
function, and there exist nonnegative functions $l(t)$ and $k(t)$ such that
$$|f(t, x)|\leq l(t)|x|^{\gamma}+k(t)$$
for all $(t,x)\in (0,+\infty)\times\mathbb{R}$, where $t^{(1-\gamma)(1-\beta)}l(t)\in
C(0,+\infty)\bigcap L^{p}_{Loc}[0,+\infty)$ and $t^{1-\beta}k(t)\in
C(0,+\infty)\bigcap L^{p}_{Loc}[0,+\infty)$. Then the fractional differential equation
(1.8) has at least one solution in $C_{1-\beta}(0,+\infty)$.
\end{cor}
\begin{proof}
Since
\begin{equation}
\begin{split}
|f(t,x)|&\leq t^{\gamma(\beta-1)}l(t)(t^{1-\beta}|x|)^{\gamma}+k(t)\\
&\leq\left(t^{\gamma(\beta-1)}l(t)+k(t)\right)\left((t^{1-\beta}|x|)^{\gamma}+1\right),\\
\end{split}
\end{equation}
then we know
$$t^{1-\beta}\left(t^{\gamma(\beta-1)}l(t)+k(t)\right)=t^{(1-\gamma)(\beta-1)}l(t)+t^{1-\beta}k(t)\in L_{Loc}^{p}[0,+\infty),$$
and if $0<\gamma<1$, then
\begin{equation}
\lim_{t\rightarrow+\infty}\frac{t}{t^{\gamma}+1}=+\infty,
\end{equation}
if $\gamma=1$, then
\begin{equation}
\lim_{t\rightarrow+\infty}\frac{t}{t^{\gamma}+1}=1.
\end{equation}
Applying Theorem 3.7, we know that fractional differential equation
(1.8) has at least one solution in $C_{1-\beta}(0,+\infty)$. Thus the proof is
complete.
\end{proof}

\begin{rem}
In Theorem 4.2 of [14], Trif presented the existence result when
$f(t,x)\leq p(t)|x|+q(t)$, where $p\in C_{\alpha}(0,+\infty)$ and
$q\in C_{1-\beta}(0,+\infty)$ with $0\leq\alpha<\beta$ and
$2\beta-\alpha>1$.

Webb [15, Theorem 4.11] proved the existence results of the equation (1.8) when nonnegative function $f(t,x)=t^{-\gamma}g(t,x)$, where $g(t,x)\leq M(1+x)$, $M>0$ and $0\leq\gamma<\beta$.

In Theorem 4.3 of [20], Zhu proved that the fractional differential equation (1.8) has at least one global solution in $C_{1-\beta}(0,+\infty)$ when $|f(t, x)|\leq l(t)|x|+k(t)$, where $t^{\beta-1}l(t)\in C(0,+\infty)\bigcap L_{Loc}^{p}[0,+\infty)$ and
$k(t)\in C(0,+\infty)\bigcap L_{Loc}^{p}[0,+\infty)$
$(p>\frac{1}{\beta}, \beta\in(0,1))$.

From Corollary 3.8, we can easily obtain that our
results include and generalize Theorem 4.2 in [14], Theorem 4.11 in [15] and Theorem 4.3 in [20]. We notice that our hypothesis in Corollary 3.8 is weaker than that imposed by Zhu [20, Theorem 4.3].
\end{rem}

\begin{thm}
Let $p>\frac{1}{\beta}$ and $q=\frac{p}{p-1}$. If $f: (0,+\infty)\times\mathbb{R}\rightarrow\mathbb{R}$ is a continuous
function and
$$|f(t,x)-f(t,y)|\leq l(t)|x-y|$$ for all $x, y\in \mathbb{R}$ and
$t\in(0,+\infty)$, where $l(t)\in C(0,+\infty)\bigcap
L_{Loc}^{p}[0,+\infty)$ and $t^{1-\beta}|f(t,0)|\in L_{Loc}^{p}[0,+\infty)$. Then the equation (1.8)
has a unique solution on $(0,+\infty)$.
\end{thm}
\begin{proof}
We know
\begin{equation}
\begin{split}
|f(t,x)|&\leq|f(t,x)-f(t,0)|+|f(t,0)|\\
&\leq l(t)|x|+|f(t,0)|.\\
\end{split}
\end{equation}
Since $l(t)\in C(0,+\infty)\bigcap
L_{Loc}^{p}[0,+\infty)$ and $t^{1-\beta}|f(t,0)|\in L_{Loc}^{p}[0,+\infty)$, applying Corollary 3.8, we know that the fractional differential equation
(1.8) has at least one solution in $C_{1-\beta}(0,+\infty)$. We suppose that $x_{1}(t)$,
$x_{2}(t)$ are two global solutions of the equation (1.8). Then
\begin{equation}
\begin{split}
|x_{1}(t)-x_{2}(t)|&=\left|\frac{1}{\Gamma(\beta)}\int_{0}^{t}(t-s)^{\beta-1}(f(s,x_{1}(s))-f(s,x_{2}(s)))ds\right|\\
&\leq\frac{1}{\Gamma(\beta)}\int_{0}^{t}(t-s)^{\beta-1}l(s)|x_{1}(s)-x_{2}(s)|ds.\\
\end{split}
\end{equation}
By (2.27) in Corollary 2.6, we can get $x_{1}(t)=x_{2}(t)$. Thus the proof is
complete.
\end{proof}

\begin{rem}
In [14, Theorem 4.1], Trif presented the uniqueness result of the fractional differential equation (1.8) when
$f(t,x)=p(t)x+q(t)$, where $p\in C_{\alpha}(0,+\infty)$ and
$q\in C_{1-\beta}(0,+\infty)$ with $0\leq\alpha<\beta$.

Webb [15, Theorem 4.11] proved the uniqueness result of the equation (1.8) when $f(t,x)=t^{-\gamma}g(t,x)$, where $g(t,x)\leq M(1+x)$ and $|g(t,x)-g(t,y)|\leq L|x-y|$, $M>0$, $L>0$ and $0\leq\gamma<\beta$.

In [20, Theorem 4.4], Zhu proved that the fractional differential equation (1.8) has a unique solution in $C_{1-\beta}(0,+\infty)$ when $|f(t, x)-f(t, y)|\leq l(t)|x-y|$, where $t^{\beta-1}l(t)\in C(0,+\infty)\bigcap L_{Loc}^{p}[0,+\infty)$ and
$|f(t,0)|\in C(0,+\infty)\bigcap L_{Loc}^{p}[0,+\infty)$
$(p>\frac{1}{\beta})$.

From Theorem 3.10, we can easily get that Theorem 3.10 includes Theorem 4.1 in [14], Theorem 4.11 in [15] and Theorem 4.4 in [20]. We notice that our hypothesis in Theorem 3.10 is weaker than that imposed by Zhu [20, Theorem 4.4].
\end{rem}

\begin{exa}
\begin{equation}
\begin{cases}
D_{r}^{\frac{2}{3}}x(t)=t^{\frac{-3}{4}}x^{\frac{1}{2}}(t)+t^{\frac{-1}{2}}x(t),\\
\lim_{t\rightarrow 0^{+}}t^{\frac{1}{3}}x(t)=1.
\end{cases}
\end{equation}
\end{exa}
We know
\begin{equation}
t^{\frac{-3}{4}}x^{\frac{1}{2}}(t)+t^{\frac{-1}{2}}x(t)\leq (t^{\frac{-11}{12}}+t^{\frac{-5}{6}})
\left((t^{\frac{1}{3}}x(t))^{\frac{1}{2}}+t^{\frac{1}{3}}x(t)\right),
\end{equation}
let $l(t)=t^{\frac{-11}{12}}+t^{\frac{-5}{6}}$, then $t^{\frac{1}{3}}l(t)=t^{\frac{-7}{12}}+t^{\frac{-1}{2}}\in L^{p}_{Loc}[0,+\infty)$($p>\frac{3}{2}$), and
\begin{equation}
\lim_{t\rightarrow+\infty}\frac{t}{t^{\frac{1}{2}}+t}=1.
\end{equation}
Applying Theorem 3.7, we know that the equation (3.22) has at least one global solution in $C_{\frac{1}{3}}(0,+\infty)$.

\begin{exa}
\begin{equation}
\begin{cases}
D_{r}^{\frac{2}{3}}x(t)=t^{\frac{-2}{3}}\ln(1+x^{\frac{1}{2}}(t)),\\
\lim_{t\rightarrow 0^{+}}t^{\frac{1}{3}}x(t)=1.
\end{cases}
\end{equation}
\end{exa}
Using Corollary 3.8, we know
\begin{equation}
t^{\frac{-2}{3}}\ln(1+x^{\frac{1}{2}}(t))\leq t^{\frac{-2}{3}}x^{\frac{1}{2}}(t),
\end{equation}
and $t^{\frac{1}{6}}l(t)=t^{\frac{-1}{2}}\in C(0,+\infty)\bigcap L^{p}_{Loc}[0,+\infty)$($p>\frac{3}{2}$), then the equation (3.25) has at least one global solution on $(0,+\infty)$.

In Example 3.13, we know $$t^{\frac{-2}{3}}\ln(1+x^{\frac{1}{2}}(t))\leq\frac{t^{\frac{-2}{3}}(x(t)+1)}{2}.$$ Thus, Theorem 4.2 in [14], Theorem 4.11 in [15] or Theorem 4.3 in [20] can not be applied to Example 3.13 (see Remark 3.9).

\begin{exa}
\begin{equation}
\begin{cases}
D_{r}^{\frac{2}{3}}x(t)=t^{\frac{-1}{2}}\frac{x^{2}(t)}{1+x(t)}+t^{\frac{-3}{4}},\\
\lim_{t\rightarrow 0^{+}}t^{\frac{1}{3}}x(t)=1.
\end{cases}
\end{equation}
\end{exa}
We know $$|\frac{x^{2}}{1+x}-\frac{y^{2}}{1+y}|\leq|x-y|, \qquad  x,y\in[0,+\infty).$$
Using Theorem 3.10, since $t^{\frac{-1}{2}}\in C(0,+\infty)\bigcap
L_{Loc}^{p}[0,+\infty)$ and $t^{\frac{-5}{12}}\in C(0,+\infty)\bigcap
L_{Loc}^{p}[0,+\infty)$ $(p>\frac{3}{2})$, then the equation (3.27) has a unique solution on $(0,+\infty)$.

Since $t^{\frac{-3}{4}}\notin C_{\frac{1}{3}}(0,+\infty)$ and $t^{\frac{-3}{4}}\notin L^{p}[0,+\infty)$($p>\frac{3}{2}$), then Theorem 4.1 in [14], Theorem 4.11 in [15] or Theorem 4.4 in [20] can not be applied to Example 3.14 (see Remark 3.11).

\bigskip

\centerline{Acknowledgements}
The research was supported by Scientific Research Foundation of Nanjing
Institute of Technology(No: CKJB201508). I would like to express my sincere gratitude to the anonymous referee for his or her comments and suggestions.

\end{document}